\documentclass[a4paper, 11pt]{amsart}

\usepackage{amsthm, amssymb}
\usepackage{amsmath}
\usepackage{bbm}
\usepackage{graphicx}
\usepackage{hyperref}
\usepackage[dvipsnames]{xcolor}
\usepackage[capitalize,nameinlink,noabbrev,nosort]{cleveref}
\usepackage[inline]{enumitem}
\usepackage{tikz}
\usetikzlibrary{positioning}
\usetikzlibrary{patterns}

\newtheorem{thm}{Theorem}[section]

\newtheorem{notn}[thm]{Notation}
\newtheorem{prop}[thm]{Proposition}
\newtheorem{lem}[thm]{Lemma}
\newtheorem{cor}[thm]{Corollary}
\newtheorem{rem}[thm]{Remark}
\newtheorem{example}[thm]{Example}

\DeclareMathOperator{\desset}{Des}
\DeclareMathOperator{\ndesset}{NDes}
\DeclareMathOperator{\des}{des}
\DeclareMathOperator{\asc}{asc}
\DeclareMathOperator{\maj}{maj}
\DeclareMathOperator{\comaj}{comaj}
\newcommand{\bij}[1]{f_{#1}}
\newcommand{\bwords}[2]{\mathcal{W}_{#1,#2}}
\newcommand{\ubwords}[2]{\underline{\mathcal{W}}_{#1,#2}}
\newcommand{\bwordsk}[1]{\mathcal{W}_{#1}}
\DeclareMathOperator{\T}{T}
\DeclareMathOperator{\F}{F}
\DeclareMathOperator{\rev}{rev}
\DeclareMathOperator{\pos}{pos}

\crefname{notn}{Notation}{Notations}

\begin{document}

\title
{A bijection on balanced words reversing both $\des$ and $\maj$}

\author{Arvind Ayyer}
\address{Arvind Ayyer, Department of Mathematics, 
Indian Institute of Science, Bangalore  560012, India.}
\email{arvind@iisc.ac.in}

\author{Naren Sundaravaradan}
\address{Naren Sundaravaradan, Get My parking, 17th Cross Rd, HSR Layout, Bengaluru 560102, India }
\email{nano.naren@gmx.com}

\date{\today}

\begin{abstract}
Balanced words on a finite alphabet are those words in which every letter of the alphabet occurs the same number of times.
The notion of descents and major index extends in a natural way to words.
It is known that the bivariate generating polynomials for descents and
major index over balanced words on the alphabet $[k]$ with $n$ occurrences each is 
palindromic, but a bijective proof has been missing even for balanced binary words.
We give an explicit bijection proving this result. 
For permutations (which are also balanced), our bijection is different from the complementation map.
We also show that for balanced binary words, this bijection simultaneously flips the ascent and comajor index as well.
\end{abstract}


\maketitle

\section{Introduction}

The combinatorics of words is a fascinating topic with a long history and many elegant results. In this work, we are concerned with the descent and major index on balanced words, in which each letter occurs the same number of times; see below for the precise definitions. 
It is known due to work of Tielker~\cite[Proposition 4.3.21]{tielker-2023}\footnote{the argument in the proof of \cite[Proposition 2.12]{carnevale-voll-2018} contains a gap. We thank Angela Carnevale for pointing this out.} 
that the bivariate generating polynomial of these statistics is
palindromic.
Carnevale--Voll conjectured~\cite[Conjectures A and B]{carnevale-voll-2018} further properties of this polynomial. Habsieger~\cite{habsieger-2020} and Carnevale~\cite{carnevale-2017} gave some partial results towards the conjectures.
The main result of this paper is a bijective proof of this bivariate palindromicity.

Permutations are a special case of balanced words and there is a rich history of beautiful bijections, one of the most famous being Foata's  bijection~\cite{foata-1968} exchanging the inversion number and the major index, both being Mahonian statistics. Loehr~\cite[Section 6.9]{loehr-2011} has extended
this bijection to words with fixed content, including balanced words.
Here, we are concerned with descents, which is a so-called Eulerian statistic.
The so-called `fundamental bijection' due to Foata~\cite{foata-schut-1970}
exchanges the number of descents and excedances.
Han~\cite{han-1992} gave a similar bijection on words with fixed content exchanging the major index and another Mahonian statistic called the $Z$-statistic, introduced by Zeilberger--Bressoud~\cite{zeilberger-bressoud-1985}
in their proof of Andrews' $q$-Dyson conjecture.

We now begin with the basic definitions and state our results.
Let $k$ and $n$ be fixed positive integers.
We will work with words over the alphabet $[k] = \{1, \dots, k\}$.
A word $w$ is said to be \emph{balanced} if each letter in the alphabet occurs
the same number of times in $w$.
Let $\bwords{n}{k}$ be the set of balanced words on $[k]$ with $n$ occurrences each, and
let $\bwordsk{k} = \cup_n \bwords{n}{k}$ be the set of all balanced words over $[k]$. It is easy to see that 
\begin{equation}
\label{bwords size}
|\bwords{n}{k}| = \binom{nk}{n, \dots, n}.
\end{equation}

For any finite word $w$, balanced or otherwise, let $\ell(w)$ be its length.
The \textit{descent set} of a word $w$ is defined by
\(
\desset(w) = \{i \in [\ell(w) - 1] \mid w_i > w_{i+1} \},
\)
and the number of \textit{descents} is $\des (w) = |\desset(w)|$. 
The \textit{major index} of $w$ is $\maj (w) = \sum_{i \in \desset(w)} i$.
We can similarly define \textit{ascents} and the \textit{comajor index}.
It is easy to see that the ascents and
descents are symmetric in $\bwordsk k$ because for $w = w_1 \dots w_{kn}$,
its complement $w' = (k+1-w_1)\dots (k+1-w_{2n})$ satisfies $\asc(w') = \des(w)$ and $\des(w') = \asc(w)$. 
It is quite clear that the minimum value of both $\des$ and $\maj$ over $\bwords
nk$
is $0$ corresponding to the unique word $w_{\min} = 1 \dots 1 2 \dots 2 k \dots k$. 
A little more work shows that the maximum values of $\des$ and $\maj$ in $\bwords nk$ are $(k-1)n$ and $n^2
  \binom{k}{2}$ respectively, and both occur for the unique word $w_{\max} = k(k-1)\cdots 1 \
  k(k-1)\cdots 1\ \cdots \ k(k-1)\cdots 1$.

The bivariate generating polynomial of $\des$ and $\maj$ on balanced words is 
\[
C_{n,k}(x, q) = \sum_{w \in \bwords nk} x^{\des(w)} q^{\maj(w)}.
\]

\begin{thm}[{\cite[Proposition 2.12]{carnevale-voll-2018}}]
\label{thm:main}
The bivariate generating polynomial of $\des$ and $\maj$ on balanced words is palindromic, i.e.
\[
C_{n, k}(x^{-1}, q^{-1}) = x^{-n(k-1)} q^{-n^2 \binom{k}{2}} C_{n,k}(x, q).
\]
\end{thm}

For example, the coefficients of $C_{2, 3}(x, q)$ can be written as the array
\[
\left(\begin{array}{rrrrrrrrrrrrr}
1 & 0 & 0 & 0 & 0 & 0 & 0 & 0 & 0 & 0 & 0 & 0 & 0 \\
0 & 2 & 5 & 6 & 5 & 2 & 0 & 0 & 0 & 0 & 0 & 0 & 0 \\
0 & 0 & 0 & 1 & 6 & 10 & 14 & 10 & 6 & 1 & 0 & 0 & 0 \\
0 & 0 & 0 & 0 & 0 & 0 & 0 & 2 & 5 & 6 & 5 & 2 & 0 \\
0 & 0 & 0 & 0 & 0 & 0 & 0 & 0 & 0 & 0 & 0 & 0 & 1
\end{array}\right),
\]
and it is clear that rotating the array by $180^\circ$ leaves it invariant.

We will construct an explicit bijection proving this result in \cref{sec:Bk_bijection}. The bijection builds on the one constructed for binary words, which we explain first in \cref{sec:B2_bijection}. We also show that the latter respects ascents and the comajor index, as well as a property we call mirror-symmetry.

An extended abstract of this preprint has appeared in the proceedings of FPSAC 2026~\cite{ayyer-sundaravaradan-2026}.

\section{Bijection on binary words}
\label{sec:B2_bijection}

A \textit{boolean array} is a word in the letters $\T$ and $\F$.\footnote{
We use $\T$ and $\F$ instead of $1$ and $0$ respectively to avoid confusion with the symbols in the alphabet.}
If $\mu$ is a boolean array, let
$\bar{\mu}$ be its \textit{complement}, i.e, the boolean array swapping the $\T$'s for $\F$'s and
vice versa in $\mu$.
Let $w \in \bwords{n}{2}$. Define the boolean array $\mu_w$ with respect to the
occurrences of $1$ in $w$ as
\[
  (\mu_w)_i =
  \begin{cases}
    \T &\text{if the } i\text{'th } 1 \text{ in } w \text{ is immediately preceded by 2},\\
    \F &\text{otherwise}.
  \end{cases}
\]
Similarly, define the boolean array $\nu_w$ with respect to the occurrences of $2$ in
$w$ as
\[
  (\nu_w)_i =
  \begin{cases}
    \T &\text{if the } i\text{'th } 2 \text{ in } w \text{ is followed by 1},\\
    \F &\text{otherwise}.
  \end{cases}
\]
Let $B_n = \{\T, \F\}^n$ be the set of boolean arrays of length $n$. 
For $\mu \in B_n$, let $n_{\T}(\mu)$ (resp.
$n_{\F}(\mu)$) be the number of $\T$'s (resp. $\F$'s) in $\mu$. 
Let $B_n^k = \{ \mu \in B_n \mid n_{\T}(\mu) = k\}$ be the set of boolean arrays having $k$ $\T$'s.
Define
\[
P_n = \bigsqcup_{0 \le k \le n} B_n^k \times B_n^k.
\]
It is clear that the cardinality of $P_n$ is given by
\[
|P_n| = \sum_{k = 0}^n \binom{n}{k}^2 = \binom{2n}{n},
\]
which is also the cardinality of $\bwords{n}{2}$.

Given $w \in \bwords n2$ with $\des(w) = d$, 
it is clear that $w$ can be uniquely written in the form
\[
w = 1^{p_1} 2^{q_1} \ 2 1 \ 1^{p_2} 2^{q_2} \ 2 1 \cdots 2 1 \ 1^{p_{d+1}}2^{q_{d+1}}.
\]
The descents are then at positions $p_1 + q_1 + 1, p_1 + q_1 + 2 + p_2 + q_2 + 1, 
\dots, p_1 + q_1 + 2 + \cdots + p_d + q_d + 1$.
Define a map $\phi: \bwords n2 \to P_n$ by setting
$\phi(w) = (\mu_w, \nu_w)$ where
  \begin{equation}
  \label{def mu}
    \mu_w = \F^{p_1}\T \F^{p_2} \T \cdots \F^{p_d} \T \F^{p_{d+1}},
  \end{equation}
and
  \begin{equation}
  \label{def nu}
    \nu_w = \F^{q_1}\T \F^{q_2} \T \cdots \F^{q_d} \T \F^{q_{d+1}}.
  \end{equation}
Clearly $\phi$ is a bijection that takes $\des(w)$ to $n_T(\mu_w) = n_T(\nu_w)$.
Now define the map $\bij{2} : \bwordsk{2} \rightarrow \bwordsk{2}$ by $\bij{2}(w) = \phi^{-1}(\bar{\mu}_w, \bar{\nu}_w)$.
The main result of this section is the following bijective proof of \cref{thm:main} for $k = 2$.

\begin{thm}
  \label{thm:B2}
  The map $f_2$ on $\bwordsk{2}$ is an involution (and hence a bijection) satisfying $\des(f_2(w)) = n -
  \des(w)$ and $\maj(f_2(w)) = n^2 - \maj(w)$ for all $w \in \bwords n{2}$.
\end{thm}

\begin{proof}[Proof of \cref{thm:B2}]
From the definitions above, it is clear that applying $f_2$ on $f_2(w)$ gives back $w$.
  Let $w' = \bij{2}(w)$. Then,
  \begin{align*}
    \des(w) + \des(w') &= n_{\T}(\mu_w) + n_{\T}(\mu_{w'})
    = n_{\T}(\mu_w) + n_{\T}(\bar{\mu}_w) = n.
  \end{align*}
  Suppose $w$ has $d$ descents at positions $a_1 < \dots < a_d$ so that
  $\maj(w) = a_1 + \cdots + a_d$. Then, clearly
  \begin{align*}
    a_i = |\{w_j \mid j < a_i, w_j = 1\}| + |\{w_j \mid j < a_i, w_j = 2\}| + 1.
  \end{align*}
  Now, let $\mu_w$ and $\nu_w$ be defined as in \eqref{def mu} and \eqref{def nu}
  respectively. Then
  \begin{align*}
   a_i = \left(\sum_{j=1}^{i-1} (p_j+1) + p_i \right) + 
   \left(\sum_{j=1}^{i-1} (q_j+1) + q_i \right) + 1
  \end{align*}
  and thus, $a_i$ is the position of the $i$'th $\T$ in $\mu_w$ plus
  the position of the $i$'th $\T$ in $\nu_w$ minus $1$.
  Said another way, we count each letter in $\mu_w$ and $\nu_w$ before the
  $i$'th $\T$ once and then add $1$, to obtain $a_i$.

  Let $w' = \bij{2}(w)$. By the definition of $\bij{2}$, $\mu_{w'} = \bar{\mu}_w$ and 
  $\nu_{w'} = \bar{\nu}_w$.
  Let $b_1 < \cdots < b_{n-d}$ be the positions of the $n-d$ descents of $w'$.
  By the same argument as above, $b_i$ counts each letter in 
  $\bar{\mu}_w$ and $\bar{\nu}_w$ before the $i$'th $\T$ once with a $1$ added.
  Equivalently, $b_i$ counts each letter in 
  $\mu_w$ and $\nu_w$ before the $i$'th $\F$ once with a $1$ added.

  Thus, when we add $\maj(w)$ and $\maj(\bij 2(w))$, we are counting 
  the $j$'th letter in both $\mu_w$ and $\nu_w$ exactly $(n-j)$ times, 
  plus an extra $1$. Therefore,  
  \begin{align*}
   \maj(w) + \maj(\bij{2}(w)) = \sum_{j=1}^n (j-1) + \sum_{j=1}^n (j-1) + n = n^2,
  \end{align*}
  completing the proof.
\end{proof}

See \cref{fig:B2} for an example of computing $\bij 2(w)$ for a word $w \in \bwords 52$.

\begin{center}
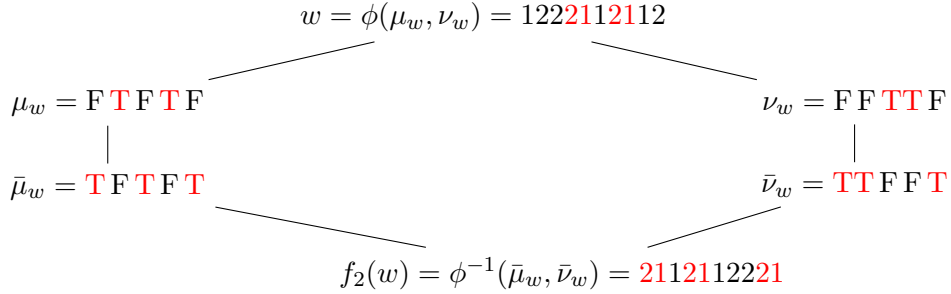
\begin{figure}[h!]
  \begin{tikzpicture}
    \node (start) {$w = \phi(\mu_w, \nu_w) = 122{\color{red}21}1{\color{red}21}12$};
    \node (mu) [below left=0.5cm and 1cm of start] {$\mu_w = \F{\color{red}\T}\F{\color{red}\T}\F$};
    \node (nu) [below right=0.5cm and 1cm of start] {$\nu_w = \F\F{\color{red}\T}{\color{red}\T}\F$};
    \node (muc) [below=0.5cm of mu] {$\bar{\mu}_w = {\color{red}\T}\F{\color{red}\T}\F{\color{red}\T}$};
    \node (nuc) [below=0.5cm of nu] {$\bar{\nu}_w = {\color{red}\T}{\color{red}\T}\F\F{\color{red}\T}$};
    \node (end) [below right=0.5cm and 1.5cm of muc] {$f_2(w) = \phi^{-1}(\bar{\mu}_w,\bar{\nu}_w) = {\color{red}21}1{\color{red}21}122{\color{red}21}$};

    \draw [-] (start) -- (mu);
    \draw [-] (start) -- (nu);
    \draw [-] (mu) -- (muc);
    \draw [-] (nu) -- (nuc);
    \draw [-] (muc) -- (end);
    \draw [-] (nuc) -- (end);
  \end{tikzpicture}
  \caption{ For $w = 1222112112$,   $\des(w) = 2$, $\maj(w) = 11$, $\des(\bij 2(w))=5-2=3$, and $\maj(\bij 2(w))
    = 25 - 11 = 14$. The $\T$'s and the descents in the words have
    both been written in red to make clear the construction of the boolean arrays.}
  \label{fig:B2}
\end{figure}
\end{center}

\subsection{Ascents}

There is a natural map on $\bwords n2$ that flips ascents and descents, and
that is the map that interchanges $1$'s and $2$'s in the word.
However, this map does not satisfy \cref{thm:B2}. However, we show
that $f_2$ satisfies the ascent equivalents of the two lemmas as well, as
illustrated in the example below.

\begin{example}
For the example $w = 1222112112$ in \cref{fig:B2}, $\asc(w) = 3$ and $\comaj(w) = 
16$. Moreover, $\asc(f_2(w)) = 5 - 3 = 2$ and $\comaj(f_2(w)) = 25 - 16 = 9$.
\end{example}

We start by defining the dual of $f_2$, which we call $\underline{f}_2$, that works on ascents instead of
descents. Given $w \in \bwords n2$ with $\asc(w) = d$, we can write $w$ uniquely
in the form
\[
  w = 2^{q_1}1^{p_1} \ 12 \ 2^{q_2}1^{p_2}\ 12 \dots 12 \ 2^{q_{d+1}}1^{p_{d+1}}.
\]
Define the map $\underline{\phi} : \bwords n2 \rightarrow P_n$ by setting $\underline{\phi}(w) =
(\underline{\mu}_w,\underline{\nu}_w)$ where
\begin{align*}
  \underline{\mu}_w &= \F^{p_1} \T \F^{p_2} \T \dots \F^{p_d} \T \F^{p_{d+1}},\\
  \underline{\nu}_w &= \F^{q_1} \T \F^{q_2} \T \dots \F^{q_d} \T \F^{q_{d+1}}.
\end{align*}
Finally, define the map $\underline{f}_2 : \bwordsk 2 \rightarrow \bwordsk 2$ by $\underline{f}_2(w) = \underline{\phi}^{-1}(\bar{\underline{\mu}}_w, \bar{\underline{\nu}}_w)$.
It should be clear that $\underline{f}_2$ is a bijection and that the proof of
\cref{thm:B2} can clearly be repurposed to prove that
$\asc(w) + \asc(\underline{f}_2(w)) = n$ and $\comaj(w) + \comaj(\underline{f}_2(w)) = n^2$ for all $w
\in \bwords n2$. We now prove something a little more surprising.

\begin{thm}
  For $w \in \bwords n2$, $\asc(w) + \asc(f_2(w)) = n$ and $\comaj(w) + \comaj(f_2(w)) = n^2$.
\end{thm}

 \begin{proof}
  Let $\mu = \mu_w, \nu = \nu_w, \underline{\mu} = \underline{\mu}_w$, and $\underline{\nu} = \underline{\nu}_w$. We prove the
  theorem by showing that complementing $(\mu, \nu)$ also complements $(\underline{\mu}, \underline{\nu})$. 
  Let
  \[
    w = 1^{p_1} 2^{q_1} \ 2 1 \ 1^{p_2} 2^{q_2} \ 2 1 \cdots 2 1 \ 
    1^{p_{d+1}}2^{q_{d+1}}
  \]
  so that
  \[
    \mu = \F^{p_1}\T \F^{p_2} \T \cdots \F^{p_d} \T \F^{p_{d+1}}
\quad \text{and} \quad
    \nu = \F^{q_1}\T \F^{q_2} \T \cdots \F^{q_d} \T \F^{q_{d+1}}.
  \]
  For any boolean array $\mu$ of length $n$, we write $L(\mu)$ to mean the word
  $\mu_2\dots \mu_n$ and $R(\mu)$ to mean the word $\mu_1\dots \mu_{n-1}$.
  If $p_1 > 0$ and $q_{d+1} > 0$ or equivalently if 
  $\mu_1 = \F$ and $\nu_n = \F$, we can also write 
  \begin{align*}
    w &= 1^{p_1-1} \ 12 \ 2^{q_1} 1^{p_2} \ 12 \ 2^{q_2} 1^{p_3} 
    \ 12 \ \dots 2^{q_d} 1^{p_{d+1}} \ 12 \ 2^{q_{d+1}-1},\\
    \underline{\mu} &= \F^{p_1-1} \T \F^{p_2} \T \dots \F^{p_{d+1}} \T \F^0 
    = L(\mu)\T,\\
    \underline{\nu} &= \F^0 \T \F^{q_1} \T \dots \F^{q_d} \T \F^{q_{d+1}-1} 
    = \T R(\nu).
  \end{align*}
  Considering all possible values of $\mu_1$ and $\nu_n$, we get the following
  for $\underline{\mu}$ and $\underline{\nu}$:
  \begin{itemize}
  \item $\mu_1 = \F, \nu_n = \T$:
  \begin{align*}
  \underline{\mu} =& \F^{p_1-1} \T \F^{p_2} \T \dots \T \F^{p_{d+1}+1} 
  = L(\mu)\F, \\
  \underline{\nu} =& \F^0 \T \F^{q_1} \T \dots \T \F^{q_d} 
  = \T R(\nu).  
  \end{align*}

  \item $\mu_1 = \T, \nu_n = \F$:
  \begin{align*}
  \underline{\mu} =& \F^{p_2} \T \dots \T \F^{p_{d+1}} \T \F^0 
  = L(\mu) \T, \\ 
  \underline{\nu} =& \F^{q_1+1} \T \dots \T \F^{q_{d+1} -1} 
  = \F R(\nu).
  \end{align*}

  \item $\mu_1 = \T, \nu_n = \T$:
  \begin{align*}
  \underline{\mu} =& \F^{p_2} \T \dots \T \F^{p_{d+1}+1} 
  = L(\mu) \F, \\
  \underline{\nu} =& \F^{q_1+1} \T \dots \T \F^{q_d} = \F R(\nu).  
  \end{align*}

  \end{itemize}

  Note that the four cases above partition $\bwords n 2$. It is also clear
    that the mapping is invertible. The proof is completed by seeing that the map is preserved under
    complementation. For
    instance, in the first case, $(\mu, \nu)$ maps to $(\underline{\mu},
    \underline{\nu}) = (L(\mu)\T, \T R(\nu))$. Then
    $(\overline\mu, \overline\nu)$, which matches the fourth case, 
    maps to  $(L(\overline\mu)\F, \F R(\overline\nu))$. 
    This is equal to $\overline{((L(\mu)\T)}, \overline{(\T R(\nu))})$, 
    which is just $(\overline{\underline{\mu}}, \overline{\underline{\nu}})$ 
    as required.
 \end{proof}

\begin{rem}
  For $k > 2$, the ascents and descents do not behave well jointly. For instance,
  $w = 111222333 \in \bwords 33$ is the only word having $\des(w) = 0, \maj(w)=
  0, \asc(w) =2$, and $\comaj(w) = 9$ but there are no words having $\des(w)=6,
  \maj(w)=27, \asc(w)=4$, and $\comaj(w)=18$.
\end{rem}

\subsection{Mirror-symmetric words}

A balanced word $w \in \bwords{n}{2}$ can be drawn as a lattice path in an $n
\times n$ grid that starts at $(0,0)$ and ends at $(n,n)$ taking north and
east steps. One way to do this is to
let a $1$ in $w$ be an east step in the lattice path and to let $2$ to be a
north step. 
We say that such a lattice path is \emph{mirror-symmetric} if it is invariant 
under a reflection across the NW-SE diagonal (i.e. the line connecting $(n, 0)$
to $(0, n)$). 
If $w = w_1 \dots w_\ell$, then let $\rev(w) = w_{\ell} \dots w_1$ be the \textit{reverse} of $w$. We can also define the \emph{complement} of a letter by
setting $1^c = 2$ and $2^c = 1$, and the complement of $w$ by
$w^c = w_1^c \dots w_\ell^c$.
It is not difficult to see that a word $w$ of length $2n$ is mirror-symmetric if and only if
\begin{equation}
\label{mirror-sym}
  w_{n+1}\dots w_{2n} = \rev((w_1 \dots w_n)^c).
\end{equation}
For example, both $w_{\min} = 1^n2^n$ and $w_{\max} = f_2(w) = (21)^n$ are
mirror-symmetric words. Since any choice of the first $n$ letters uniquely determines a mirror-symmetric word of length $2n$ by \eqref{mirror-sym},
it follows that there are $2^n$ mirror-symmetric words in $\bwords n2$.

\begin{thm}
\label{thm:mirror}
  Let $w \in \bwordsk{2}$ be mirror-symmetric. Then so is $f_2(w)$.
\end{thm}

\begin{proof}
  Let $w$ a mirror-symmetric word in $\bwords{n}{2}$. 
The idea is to consider the structure of $\mu_w$ and $\nu_w$.
 Write $w = h\rev(h)^c$ for some word $h$. We
  consider two cases. If $\des(w) = 2d$, then
  \[
    h = 1^{p_1}2^{q_1} \ 21 \ \cdots 21 \ 1^{p_{d+1}} 2^{0},
  \]
  and so
  \[
    \rev(h)^c = 1^{0}2^{p_{d+1}} \ 21 \ \cdots 21 \ 1^{q_1}2^{p_1}.
  \]
  Thus, $w$ is mirror-symmetric if and only if
  \begin{align*}
    \mu_w &= \F^{p_1}\T \dots \T \F^{p_{d+1}} \T \F^{q_d} \T \dots \T \F^{q_1}\\
    \nu_w &= \F^{q_1}\T \dots \T \F^{p_{d+1}} \T \F^{p_d} \T \dots \T \F^{p_1},
  \end{align*}
  implying $\mu_w = \rev(\nu_w)$. Now suppose $\des(w) = 2d+1$. Then,
  \[
    h = 1^{p_1}2^{q_1}(21) \dots (21)1^{p_d}2^{q_d}2
  \]
  and
  \[
    \rev(h)^c = 1 1^{q_d}2^{p_d}(21)\dots(21)1^{q_1}2^{p_1}.
  \]
  Thus, again $w$ is mirror-symmetric if and only if
  \begin{align*}
    \mu_w &= \F^{p_1}\T \dots \T \F^{p_d} \T \F^{q_d} \T \dots \T \F^{q_1}\\
    \nu_w &= \F^{q_1}\T \dots \T \F^{q_d} \T \F^{p_d} \T \dots \T \F^{p_1},
  \end{align*}
  again implying $\mu_w = \rev(\nu_w)$.
  Therefore, in either case, 
  $\mu_{f_2(w)} = \bar{\mu}_w = \overline{\rev(\nu_w)} = \rev(\nu_{f_2(w)})$,
  and hence $f_2(w)$ is also  mirror-symmetric .
\end{proof}

\begin{figure}[h!]
  \centering
  \begin{tikzpicture}[scale=0.5]
    \draw[help lines] (0,0) grid +(10,10);
    \draw[dashed] (0,10) -- (10,0);
    \draw[line width=1pt] (0,0) -- (1,0) -- (1,1) -- (1,2) -- (1,3) -- (1,4) -- (2,4) -- (2,5) -- (2,6) -- (2,7) -- (3,7) -- (3,8) -- (4,8) -- (5,8) -- (6,8) -- (6,9) -- (7,9) -- (8,9) -- (9,9) -- (10,9) -- (10,10);
    \draw[line width=1pt,color=blue] (0,0) -- (0,1) -- (1,1) -- (2,1) -- (3,1) -- (4,1) -- (4,2) -- (5,2) -- (5,3) -- (6,3) -- (7,3) -- (7,4) -- (7,5) -- (8,5) -- (8,6) -- (9,6) -- (9,7) -- (9,8) -- (9,9) -- (9,10) -- (10,10);
  \end{tikzpicture}
  \caption{The black lattice path corresponds to the mirror-symmetric word 
  $w$ in \cref{eg:mirror}. The blue path corresponds to 
  $f_2(w)$ and is also clearly mirror-symmetric.}
  \label{fig:mirror}
\end{figure}
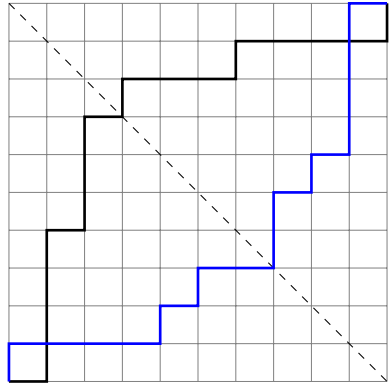

\begin{example}
\label{eg:mirror}
The word $w = 12222122212111211112$ is mirror-symmetric as seen in \cref{fig:mirror}. 
One can verify that
$f_2(w) = 21111212112212122221$ is also mirror-symmetric. 
Note that $f_2(w)$ is not obtained from $w$ by interchanging $1$'s and $2$'s.
  
\end{example}

\section{Bijection on $\bwordsk{k}$}
\label{sec:Bk_bijection}

We will construct a bijective map $f_k$ on $\bwords nk$ to prove \cref{thm:main} in this section. The map will be constructed inductively using $f_{k-1}$.

\subsection{Various maps}
Here, we will define the maps needed to construct the bijection.
We will consider the running example $w = 325544135121432 \in \bwords 35$ throughout this section to illustrate the notation and ideas. 

Recall that in \cref{sec:B2_bijection}, a word $w \in \bwords{n}{2}$ is reversibly 
encoded as a pair of boolean arrays $(\mu_w,\nu_w)$. 
For $w \in \bwords nk$, define the word $w_{-k} \in \bwords{n}{k-1}$ as the subword of $w$ with the $k$'s deleted. 
We will reversibly encode $w$ as a triplet $(w_{-k}, \mu_w, \nu_w)$ for some boolean arrays $\mu_w$ and $\nu_w$.

\begin{notn}
\label{notn:w_underline}
  For $w \in \bwords{n}{k}$, let $\underline{w}$ be the subword of $w$ 
  obtained by removing all occurrences of $k$ which do not lead to a change
  in the number of descents. 
  In particular, let $w_i$ to $w_j$ be a sequence of
  $k$'s for $1 \le i \le j \le n$. Also let $w_i=0$ for $i \le 0$ and $i >
  n$. Then, delete $w_i \dots w_j$ if one of the following is true:
  \begin{enumerate}
  \item $w_{i-1} > w_{j+1}$ and $w_{i-1}, w_{j+1} < k$, or
  \item $w_{i-1} \le w_{j+2}$, $w_{j+1}=k$, and $w_{i-1},w_{j+2} < k$.
  \end{enumerate}
\end{notn}

It is clear from the construction that the order of deletion in various blocks is irrelevant.
For example, if $w =113323112223 \in \bwords 43$, we would remove the $3$ at positions $6$ and $12$ due to $(1)$, and the $3$ at position $3$ due to $(2)$. 
So, $\underline{w} = 113211222$. 

By construction,
$\des(\underline{w}) = \des(w)$.
Define the set of \textit{non-descents}
of a (not necessarily balanced) word $\ndesset(w)$ in $w$
as
\begin{equation}
\label{ndesset}
 \ndesset(w) = \{ 0 \} \cup \{ 1 \le i \le \ell(w)-1 \mid w_i \le w_{i+1}\}.
\end{equation}
The size of $\ndesset(w)$ is $|\ndesset(w)| = \ell(w) - \des(w)$.
For our running example, we obtain $\underline{w} = 3254413121432$ and
$\ndesset(w_{-5}) = \{0, 2, 3, 5, 7, 9\}$.

\begin{lem}\label{lem:index_set_I}
  For $w \in \bwords{n}{k}$, suppose the positions of $k$ in $\underline{w}$ are
  $a_1 < \dots < a_r$.
  Then $\{a_i - i \mid 1 \le i \le r\} \subseteq \ndesset(w_{-k})$.
\end{lem}

\begin{proof}
  Note that $a_i - i$ is the position, in $w_{-k}$, of the letter preceding the
  $i$'th $k$ in $\underline{w}$. The definition of $\underline{w}$ ensures that
  it has no consecutive $k$'s. By \cref{notn:w_underline}, the removal
  of any $k$ that appears in $\underline{w}$ must necessarily decrease
  descents by $1$. Therefore, the positions $a_i - i$ in $w_{-k}$ are
  non-descents.
\end{proof}

Let $\bwords{n}{k}^d = \{ w \in \bwords{n}{k} \mid \des(w) = d\}$ be the set of
balanced words having $d$ descents.

\begin{notn}\label{notn:general_mu}
  For $w \in \bwords{n}{k}^d$, let $a_1 < \dots < a_r$ be the
  positions of $k$ in $\underline{w}$ and 
  $\ndesset(w_{-k}) = \{b_1 < \dots < b_s\}$. 
  Denote by $\mu_w \in B_{(k-1)n - d + r}^r$ the boolean array given by
  \[
    (\mu_w)_j =
    \begin{cases}
      \T &\text{if } b_j \in \{a_i - i \mid 1 \le i \le r\} ,\\
      \F &\text{otherwise}.
    \end{cases}
  \]
  Note that $\ell(\mu_w) = |\ndesset(w_{-k})| = (k-1)n - d + r$.
\end{notn}

In our running example, $w \in \bwords 3{5}^7$, we have $r = 1, a_1 = 3$ and so $\mu_w = \F\T\F\F\F\F \in B^1_6$. 

For any subset $W
\subseteq \bwords{n}{k}$, let $\underline{W} = \{ \underline{w} \mid w \in W \}$.
Also, let $\bwords{n}{k}^{d_1,d_2}$ be the subset of $\bwords{n}{k}^{d}$ such that $d_1 = \des(w_{-k})$ for $w \in \bwords{n}{k}^{d}$ and $d_2 = d - d_1$.
Then, the set $\bwords{n}{k}$ is partitioned as
\[
\bwords{n}{k} = \bigsqcup_{\substack{0 \le d_1 \le (k-2)n \\ 0 \le d_2 \le n}}
\bwords{n}{k}^{d_1,d_2}.
\]

\begin{notn}\label{notn:map_g}
For $u \in \bwords{n}{k-1}^{d}$, write  
$\ndesset(u) = \{b_1 < \dots < b_{n(k-1)-d}\}$.
  Let $g : \bwords{n}{k-1}^{d_1} \times B_{(k-1)n-d_1}^{d_2} \rightarrow
  \ubwords{n}{k}^{d_1,d_2}$ be given by $g(u, \mu) = v$ where $v$ is
constructed from $u$ by placing one $k$ after every $u_{b_i}$ whenever $\mu_i = \T$.
\end{notn}

Let $u = 324413121432 \in \bwords 3{4}^6$ and $\mu = \F\T\F\F\F\F$.
Then $v = 3254413121 \allowbreak 432$ which is exactly $\underline{w}$ from our running example.
It should be clear from \cref{lem:index_set_I} that the following proposition
is true.

\begin{prop}
The map $g$ defined in \cref{notn:map_g} is a bijection. In particular,
given $u \in \ubwords{n}{k}^{d_1,d_2}$, $g^{-1}(u) = (u_{-k}, \mu_u)$. 
\end{prop}

Next, we define the \textit{modified descent set} $\desset_+(w)$ for any word $w$ 
as $\desset_+(w) = \desset(w) \cup \{ \ell(w) \}$, whose size is $|\desset_+(w)| = \des(w) + 1$.

\begin{notn}
\label{notn:general_nu}
  Let $w \in \bwords{n}{k}^{d_1,d_2}$ with $d=d_1+d_2$, and let $\desset_+(\underline{w}) = \{j_1, \dots,
  j_{d+1}\}$. We can write $w$ in terms of $\underline{w}$ as
  \[
    w = \underline{w}_1 \dots \underline{w}_{j_1}k^{p_1} \underline{w}_{j_1+1}    \dots \underline{w}_{j_2}k^{p_2} \underline{w}_{j_2+1} \dots \underline{w}_{kn-d_2}k^{p_{d+1}}
  \]
  for a unique assignment of integers $p_1, \dots, p_{d+1} \ge 0$ that add up to $n-d_2$. Now,
  let $\nu_w \in B_{n+d_1}^{n-d_2}$ be
  \[
    \nu_w = \T^{p_1}\F \T^{p_2} \dots \F\T^{p_d} \F \T^{p_{d+1}}.
  \]
\end{notn}

From our running example, we see that $\desset_+(\underline{w}) =
\{1,3,5,7,9,11,12,13\}$ and $\nu_w = \F\T\F\F\T\F\F\F\F$.

\begin{notn}\label{notn:map_h}
Write $\nu \in B_{n+d_1}^{n-d_2}$ as $\nu = \T^{p_1}\F
  \T^{p_2} \dots \T^{p_d} \F \T^{p_{d+1}}$.
  Then define $h : \ubwords{n}{k}^{d_1,d_2} \times B_{n+d_1}^{n-d_2} \rightarrow
  \bwords{n}{k}^{d_1,d_2}$ by $h(u, \nu) = v$, where $v$ is constructed
  from $u$ by placing $p_j$ $k$'s in $u$ immediately after each position $j \in
  \desset_+(u)$.
\end{notn}

Applying $h(3254413121432,\F\T\F\F\T\F\F\F\F)$ gives back our running example 
$w=325544135121432$.

It should be clear from \cref{notn:map_h} that the following result is true.

\begin{prop}
The map $h$ defined in \cref{notn:map_h} is a bijection. In particular,
given $u \in \bwords{n}{k}^{d_1,d_2}$,
$h^{-1}(u) = (\underline{u}, \nu_u)$.
\end{prop}

\begin{notn}
\label{notn:map_gamma}
For $k \ge 3$, $n \ge 1$, $0 \le d_1 \le (k-2)n$, and
$0 \le d_2 \le n$, let
\[
R_{n,k}^{d_1,d_2} = \bwords{n}{k-1}^{d_1} \times B_{(k-1)n - d_1}^{d_2} \times B_{n+d_1}^{n-d_2} \quad
\text{and} \quad
R_{n,k} = \bigsqcup_{\substack{0 \le d_1 \le (k-2)n \\ 0 \le d_2 \le n}}
R_{n,k}^{d_1,d_2}.
\]
Now let $\Gamma_k : \bwords nk \rightarrow R_{n,k}$ be defined by $\Gamma_k(w) = (w_{-k}, \mu_w, \nu_w)$.
\end{notn}

\begin{lem}
\label{lem:Gamma}
  The function $\Gamma_k$ is a bijection.
\end{lem}

\begin{proof}
Suppose $\Gamma_k(w) = (w_{-k}, \mu_w, \nu_w)$. 
Then, if $w \in \bwords n{k}^d$, 
let $d_1 = \des(w_{-k})$ and $d_2 = d - d_1$. 
Then $\mu_w \in  B_{(k-1)n - d_1}^{d_2}$, and so $g(w_{-k},\mu_w)$ is well-defined and lives in $\ubwords{n}{k}^{d_1,d_2}$. 
Moreover, $\nu_w \in  B_{n + d_1}^{n - d_2}$ and therefore, $h(g(w_{-k},\mu_w), \nu_w)$ is well-defined and belongs to $\bwords{n}{k}^{d}$. One then verifies that $w = h(g(w_{-k},\mu_w), \nu_w)$.
\end{proof}

\cref{lem:Gamma} gives another proof of the identity.
\[
  |\bwords nk| = \sum_{d_1 = 0}^{(k-2)n} |\bwords n{k-1}^{d_1}| 
  \sum_{d_2 = 0}^{n} \binom{(k-1)n - d_1}{d_2} \binom{n+d_1}{n-d_2}.
\]
A more direct way to see this is to use the Chu--Vandermonde identity to conclude that the inner sum gives $\binom{kn}{n}$, which does not depend on $d_1$. Then, using \eqref{bwords size} proves the result.

\subsection{Proof of the main result}

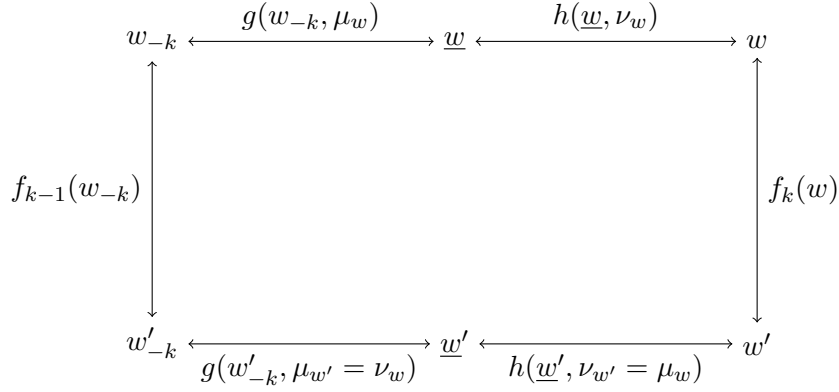
\begin{figure}[h!]
  \centering
  \begin{tikzpicture}[node distance=4cm]
    \node (start)                   {$w_{-k}$};
    \node (mid)    [right of=start] {$\underline{w}$};
    \node (end)    [right of=mid]   {$w$};
    \node (start2) [below of=start] {$w'_{-k}$};
    \node (mid2)   [below of=mid]   {$\underline{w}'$};
    \node (end2)   [below of=end]   {$w'$};

    \draw [<->] (start) -- (mid) node [above,text centered,midway] {$g(w_{-k},\mu_w)$};
    \draw [<->] (mid) -- (end) node [above,text centered,midway] {$h(\underline{w}, \nu_w)$};
    \draw [<->] (start) -- (start2) node [left,text centered,midway] {$\bij{k-1}(w_{-k})$};
    \draw [<->] (end) -- (end2) node [right,text centered,midway] {$\bij{k}(w)$};
    \draw [<->] (start2) -- (mid2) node [below,text centered,midway] {$g(w'_{-k}, \mu_{w'} = \nu_{w})$};
    \draw [<->] (mid2) -- (end2) node [below,text centered,midway] {$h(\underline{w}',\nu_{w'} = \mu_w)$};
  \end{tikzpicture}
  \caption{Schematic for the construction of $\bij{k}$ for $k \ge 3$. 
  In particular, $f_k(w) = \Gamma_k^{-1}((f_{k-1}(w_{-k}), \nu_w, \mu_w))$.}
  \label{fig:schematic}
\end{figure}

The idea of the proof is given in \cref{fig:schematic}. 
Starting with $w \in \bwords nk$, we obtain $\underline{w}$ and $w_{-k}$ as well as $\mu_w$ and $\nu_w$. We then apply $f_{k-1}$ to get $w'_{-k}$. Finally, we use the $g$ and $h$ maps along with $\mu_{w'} = \nu_{w}$ and $\nu_{w'} = \mu_w$ to construct the word $w' \in \bwords nk$.
Formally, define the map $f_k : \bwords nk \rightarrow \bwords nk$ for $k \ge 3$ by
\[
  f_k(w) = \Gamma_k^{-1}((f_{k-1}(w_{-k}), \nu_w, \mu_w)).
\]
The base case of this composite map is the function $f_2$ defined in \cref{sec:B2_bijection}.
This is the main result.

\begin{thm}\label{thm:Bk_major_index}
  For $k \ge 3$, $f_k$ satisfies 
  $\des(w) + \des(f_k(w)) = (k-1)n$ and
  $\maj(w) + \maj(f_k(w)) = n^2\binom{k}{2}$ for all $w \in \bwordsk{k}$.
\end{thm}

For our running example, one can check that $f_{4}(w_{-5}) = 142422331134$. Recalling that $\mu_w = \F\T\F\F\F\F$ and $\nu_w = \F\T\F\F\T\F\F\F\F$, we get  $\underline{f_{5}(w)} = 15424225331134$ and $f_{5}(w) = 154524225331134$.
It is easily verified that $\des(f_{5}(w)) = 12 - 7 = 5$ and $\maj(f_{5}(w)) = 90 - 58 = 32$.
It might help to see how $\Gamma$ and $f$ act on all of $S_3$ in \cref{tab:S3}.

\begin{table}[h!]
\begin{center}
\[
\begin{array}{|cccccccc|}
\hline
w \in S_3 & \mu_w & w_{-3} & g(w_{-3}, \mu_w) & \nu_w & h(g(w_{-3},\mu_w)) 
& \des & \maj\\
\hline
123 & \F\F & 12 & 12 & \T & 123 & 0 & 0\\
132 & \F\T & 12 & 132 & \F & 132 & 1 & 2\\
312 & \T\F & 12 & 312 & \F & 312 & 1 & 1\\
\hline
321 & \T & 21 & 321 & \F\F & 321 & 2 & 3\\
213 & \F & 21 & 21 & \F\T & 213 & 1 & 1\\
231 & \F & 21 & 21 & \T\F & 231 & 1 & 2\\
\hline
\end{array}
\]
\caption{The action of the map $\Gamma_3$ on $S_3$. The permutations are ordered so that $f_3$ maps the $i$'th row above the bar to the $i$'th row below the bar. Notice that the roles of $\mu_w$ and $\nu_w$ get swapped in these rows.}
\label{tab:S3}
\end{center}
\end{table}

For the proof of \cref{thm:Bk_major_index}, we introduce a little notation. For
a boolean array $\mu$, let $\overrightarrow{\mu}(i | \F)$ be the number of $\F$
occurring strictly to the right of the $i$'th $\T$ in $\mu$. We also define the
variants $\overleftarrow{\mu}(i | \F)$, $\overrightarrow{\mu}(i |
\T)$, and $\overleftarrow{\mu}(i | \T)$ similarly, always with respect to the $i$'th $\T$.
We also define $\pos_\mu(i)$ to be the position of the $i$'th $\T$ in $\mu$.

\begin{lem}
\label{lem:maj_reminder}
  Let $w \in \bwords{n}{k}$ for $k \ge 3$. 
Then,
  \[
    \maj(w) - \maj(w_{-k}) + \maj(f_k(w)) - \maj(f_k(w)_{-k}) = (k-1)n^2.
  \]
\end{lem}

\begin{proof}
  Let $w \in \bwords{n}{k}^{d_1,d_2}$. 
  For brevity, we write $f$ for $f_k$, $\mu = \mu_w$ and $\nu = \nu_w$. We
  will split the contribution of $k$ to the major index into two parts.
  We first look at $\maj(\underline{w}) - \maj(w_{-k})$.
  Here, we are concerned with the contribution of each $k$ in $\underline{w}$ to
  the major index and we want to capture this in terms of $\mu$. 
  For $i \in [d_2]$, let $a_i$ be the position of the $i$'th $k$ in 
  $\underline{w}$ and let $b_i$ be the number of descents in $w_{-k}$ 
  beginning at position $a_i-(i-1)$ respectively.
  Then
  \begin{align*}
    \maj(\underline{w}) - \maj(w_{-k}) &= \sum_{i=1}^{d_2} (a_i + b_i)\\
    &= d_2\des(w_{-k}) + \sum_{i=1}^{d_2} (a_i - (\des(w_{-k}) - b_i)).
  \end{align*}
  Now, $\des(w_{-k}) - b_i$ is the number of descents before the $i$'th descent in 
  $w_{-k}$. Thus, $a_i - (\des(w_{-k}) - b_i)$ counts one more than 
  the number of non-descents
  in $w_{-k}$ plus the number of $k$'s before the $i$'th $k$ in 
  $\underline{w}$. Expressed in terms of $\pos_\mu(i)$
  and $\overleftarrow{\mu}(i | \T)$, we thus get
  \[
  \maj(\underline{w}) - \maj(w_{-k}) =
  d_2\des(w_{-k}) + \sum_{i=1}^{d_2} \left(\pos_\mu(i) + \overleftarrow{\mu}(i|\T) \right).
  \]

  Second, we compute the contribution of the $k$'s not in $\underline{w}$. Let
  $r_1 < \dots < r_{n-d_2}$ be the positions of these $k$ in $w$
  and let
  $s_1,\dots,s_{n-d_2}$ be the descents in $\underline{w}$ occurring 
  starting at positions
  $r_1 - 1, \dots, r_{n - d_2} - (n - d_2)$ respectively. Because
  of the shift of positions of these descents, it is clear that
  \[
    \maj(w) - \maj(\underline{w}) = \sum_{i=1}^{n-d_2} s_i = \sum_{i=1}^{n-d_2} 
    \overrightarrow{\nu}(i | \F).
  \]
  and so
  \[
    \maj(f(w)) - \maj(f(\underline{w})) = \sum_{i=1}^{n-d_2} 
    \overrightarrow{\mu}(i | \F).
  \]  
  Putting the two components together, we get
  \begin{multline*}
   \maj(\underline{w}) - \maj(w_{-k}) + \maj(f(w)) - \maj(\underline{f(w)})\\
    = d_2\des(w_{-k}) + \sum_{i=1}^{d_2} \left(\pos_\mu(i) 
    + \overleftarrow{\mu}(i | \T) + \overrightarrow{\mu}(i | \F) \right).
  \end{multline*}
  Now, it is clear that $\sum_i \overleftarrow{\mu}(i | \T) = 
  \sum_i \overrightarrow{\mu}(i | \T)$. Plugging that in, the summand 
  becomes $\pos_\mu(i) + \overrightarrow{\mu}(i | \T) 
  + \overrightarrow{\mu}(i | \F)$,  which is exactly $\ell(\mu)$. Thus,
  \begin{multline*}
    \maj(\underline{w}) - \maj(w_{-k}) + \maj(f(w)) - \maj(\underline{f(w)})
    = d_2(\des(w_{-k}) + \ell(\mu))\\
    = d_2(\des(w_{-k}) + (k-1)n - \des(w_{-k}))
    = d_2(k-1)n.
  \end{multline*}
  By a similar argument, we also have
  \begin{align*}
    \maj(\underline{f(w)}) - \maj(f(w)_{-k}) + \maj(w) - \maj(\underline{w})
    = (n-d_2)(k-1)n.
  \end{align*}
  Adding these two gives the desired result.
\end{proof}

\begin{proof}[Proof of \cref{thm:Bk_major_index}]
We will prove this by induction on $k$. For $k = 2$, the result is proved in 
\cref{thm:B2}. So suppose $k > 2$ and let $w \in \bwords{n}{k}^{d_1,d_2}$. Then $\des(w) = d_1+d_2$,
  $\Gamma_k(w) = (w_{-k}, \mu_w, \nu_w)$ from \cref{notn:map_gamma}, and
  $f_{k-1}(w_{-k}) \in \bwords{n}{k-1}^{(k-2)n-d_1}$ by the induction 
  hypothesis.
  \[
    (f_{k-1}(w_{-k}), \nu_w, \mu_w) \in R_{n,k}^{(k-2)n-d_1,n-d_2} = 
    \bwords{n}{k-1}^{(k-2)n-d_1} \times B_{n + d_1}^{n-d_2} 
    \times B_{(k-1)n - d_1}^{d_2}
  \]
  from \cref{notn:general_mu,notn:general_nu}.
  By \cref{lem:Gamma}, we can define $f_k(w) = \Gamma_k^{-1}((f_{k-1}(w_{-k}), \nu_w, \mu_w))$ and $f_k(w) \in \bwords{n}{k}^{(k-2)n - d_1, n-d_2}$.  
  Therefore, 
  \[
  \des(f_k(w)) = (k-1)n - d_1 - d_2,
  \]
  proving the first part.

  Using \cref{lem:maj_reminder} and by the induction hypothesis,
  \begin{align*}
    \maj(w) + \maj(f(w)) = &\maj(w_{-k}) + \maj(f(w)_{-k}) \\
      &+\maj(w) - \maj(w_{-k}) + \maj(f(w)) - \maj(f(w)_{-k})\\
    = &\binom{k-1}{2}n^2 + (k-1)n^2 = \binom{k}{2}n^2,
  \end{align*}
  proving the second part.
\end{proof}

An illustration of \cref{thm:Bk_major_index} is given in \cref{tab:bijection_example}.

\begin{table}[h!]
  \centering
\begin{tabular}{|c|c|c|c|c|}
\hline
& $f_5$ & $f_4$ & $f_3$ & $f_2$ \\
\hline
$w$
  & $325544135121432$
  & $324413121432$
  & $321312132$
  & $211212$ \\

$\desset(w)$
  & $\{1,4,6,9,11,13,14\}$
  & $\{1,4,6,8,10,11\}$
  & $\{1,2,4,6,8\}$
  & $\{1,4\}$ \\

$\mu$
  & $\F\T\F\F\F\F$
  & $\F\F\F\T$
  & $\T\T\F\T$
  & $\T\F\T$ \\

$\nu$
  & $\F\T\F\F\T\F\F\F\F$
  & $\F\T\T\F\F\F\F\F$
  & $\F\F\F\F\F$
  & $\T\T\F$ \\

$f(w)$
  & $154524225331134$
  & $142422331134$
  & $122233113$
  & $122211$ \\

$\desset(f(w))$
  & $\{2,4,6,9,11\}$
  & $\{2,4,8\}$
  & $\{6\}$
  & $\{4\}$ \\
\hline
\end{tabular}  
\medskip
  \caption{Example of recursively computing $f_5(325544135121432)$.
  The descents of $w$ and $f(w)$ are written for convenience.}
  \label{tab:bijection_example}
\end{table}

From the proof, we immediately deduce the following.

\begin{cor}
\label{cor:proj}
The bijections satisfy the property $f_{k-1}(w_{-k}) = f_k(w)_{-k}$ for all $w \in 
\bwords nk$ for all $n \geq 1, k \geq 3$.
\end{cor}

\subsection{Permutations}

Note that permutations are a special case of balanced words with $n = 1$. Thus, one can study the map $f_k$ on the set of permutations, $S_k = \bwords 1k$. 
We note that the complementation map sending each letter $a \mapsto k  + 1 - a$ is also sufficient to bijectively prove \cref{thm:main}.
However, unlike $f_k$, this map does not reduce to $S_{k-1}$ correctly, 
i.e. does not satisfy \cref{cor:proj}.
For example, \cref{tab:permutation} shows that
 $f_8(46782513) = 76158342$ and $f_7(4672513) = 7615342$. 
We now show that $f_k$ inverts the set of descents, just like complementation.

\begin{thm}
\label{thm:perm_descent_set}
  Let $w \in S_k$. Then $\desset(f_k(w)) = [k-1] \setminus \desset(w)$.
\end{thm}
\begin{proof}
 This is clearly the case when $k=2$. Now suppose $k > 2$. Note that $n_{\T}(\mu_w) 
 + n_{\T}(\nu_w) = 1$. Without loss of generality, suppose $n_{\T}(\mu_w)=1$ and
 let $w_i = k$ (if not, we can work with $f_k(w)$ instead). Let $u = w_{-k}$, $u'= f_{k-1}(w_{-k})$, and $w'
 = f_k(w)$. We have
 \[
   \desset(w) = \desset(u_1\dots u_{i-1}) \cup \{i-1 + d \mid d \in \desset(k
   u_{i}\dots u_{k-1})\}.
 \]
 If $\T$ is at position $t$ in $\mu_w$, $w$ is constructed from $u$ by 
 inserting $k$ after the $(t-1)$'th
 non-descent (we insert at the beginning when $t=1$). On the other hand, let
 $j$ be the position of $k$ in $w'$, then we
 construct $w'$ from $u'$ by inserting $k$ after the
 $t$'th descent in $u'$ since $\nu_{f_k(w)} = \mu_w$. However, by
 our inductive assumption, the position $j-1$ in $u$ is the location of its $t$'th
 non-descent. 
 Hence, we can deduce that $i < j$, $\desset(u_i \dots u_j) = \{1, \dots,
 j-i-1\}$, and $\desset(u'_i \dots u'_j) = \emptyset$. Thus,
 \begin{align*}
   \desset(w') &= \desset(u'_1\dots u'_{j-1}) \cup \{j-1+d \mid d \in \desset(ku'_j
                 \dots u'_n)\}\\
               &= \desset(u'_1\dots u'_{i-1}) \cup \{i-1+d \mid d \in \desset(u'_i\dots u'_j ku'_{j+1}
                 \dots u'_n)\}\\
               &= \desset(u'_1\dots u'_{i-1}) \cup \left( \{i-1+d \mid d \in \desset(ku'_i\dots u'_n)\} \setminus \{i\} \right)\\
               &= \{1,\dots,k-1\} \setminus \desset(w).
 \end{align*}
\end{proof}

 \begin{table}[h!]
  \centering
  \begin{tabular}{|c|c|c|c|c|c|c|c|}
\hline
& $f_8$ & $f_7$ & $f_6$ & $f_5$ & $f_4$ & $f_3$ & $f_2$ \\
\hline
$w$
  & $46872513$
  & $4672513$
  & $462513$
  & $42513$
  & $4213$
  & $213$
  & $21$ \\

$\desset(w)$
  & $\{3,4,6\}$
  & $\{3,5\}$
  & $\{2,4\}$
  & $\{1,3\}$
  & $\{1,2\}$
  & $\{1\}$
  & $\{1\}$ \\

$\mu$
  & $\F\F\T\F\F$
  & $\F\F\F\F$
  & $\F\F\F$
  & $\F\F$
  & $\T\F$
  & $\F$
  & $\T$ \\

$\nu$
  & $\F\F\F$
  & $\T\F\F$
  & $\T\F\F$
  & $\F\T\F$
  & $\F\F$
  & $\F\T$
  & $\T$ \\

$\desset(f(w))$
  & $\{1,2,5,7\}$
  & $\{1,2,4,6\}$
  & $\{1,3,5\}$
  & $\{2,4\}$
  & $\{3\}$
  & $\{2\}$
  & $\emptyset$ \\

$f(w)$
  & $76158342$
  & $7615342$
  & $615342$
  & $15342$
  & $1342$
  & $132$
  & $12$ \\
\hline
\end{tabular}
\medskip
  \caption{Example of applying $f_8$ to the permutation $46782513$.}
  \label{tab:permutation}
 \end{table}

In the case of permutations, we get a simpler way to construct
  $f_n(w)$. 
  The definitions of $\ndesset$ in \eqref{ndesset} and $\desset_+$, 
  as well as \cref{thm:perm_descent_set} lead to the following observations.
  \begin{enumerate}
  \item   If $n$ is right after the $i$'th non-descent in $w_{-n}$,
  then the position of $n$ in $f_n(w)$ is one more than the position of the 
  $(i+1)$'th non-descent in $w_{-n}$.
  If $i$ is the last non-descent, then $n$ is at the end of $f_n(w)$.

  \item If $n$ is right after the $i$'th descent in $w_{-n}$, then
  the position of $n$ in $f_n(w)$ is one more than the position of 
  the $(i-1)$'th descent in $w_{-n}$. If $i$ is the first descent, then 
  $n$ is at the beginning of $f_n(w)$.
  \end{enumerate}

As an example, let us construct $f_8(46782513)$, the example in \cref{tab:permutation}. 

\begin{example}
Start with eight blanks $--------$.
  We see that $8$ occurs after the third non-descent in $w_{-8} = 4672513$,
  where we recall from \eqref{ndesset} that $0$ is the first non-descent.   
  Since the fourth non-descent is at position 4, $8$ must be added to the fifth blank: $----8---$. 
  Next, ignoring $8$ in $w$, $7$ occurs after the
  first descent in $w$ with $7,8$ removed. Hence, it must occur at the first blank: $7---8---$. Likewise,
  ignoring elements larger than $5$ in $w$, $6$ occurs after the first descent. Hence, it must occur at the first of the remaining blanks:
  $76--8---$. Next, $5$ occurs at the second descent. Hence, it must occur at
  the second blank: $76-58---$. 
Continuing this way, we get $f_8(46872513) = 76158342$.
\end{example}

\section{Discussion}

In this work, we have given a bijection on balanced words that simultaneously flips both the number of descents and the major index. The bijection is recursive on the size of the alphabet. Our bijection has many nice properties. For example, for balanced binary words, it also flips the number of ascents and the comajor index, as well as preserving mirror-symmetric words. For permutations, the bijection behaves nicely with respect to the restriction map.

We note that the inversion generating function of balanced words is also palindromic and has an easy bijective proof. However, our bijection does not behave nicely with inversions.

\section*{Acknowledgements}
We thank Angela Carnevale for many useful discussions and references to the literature. We also thank the anonymous reviewers for many useful comments.

\bibliographystyle{alpha}
\bibliography{main}

\end{document}